\documentclass[11pt]{article}
\usepackage[margin=1in]{geometry}
\usepackage[round]{natbib}
\usepackage{amsmath,amssymb,amsthm,bm,booktabs,graphicx,enumitem}
\usepackage[colorlinks,citecolor=blue,urlcolor=blue,linkcolor=blue]{hyperref}
\theoremstyle{plain}
\newtheorem{theorem}{Theorem}\newtheorem{proposition}{Proposition}
\newtheorem{corollary}{Corollary}
\theoremstyle{definition}
\newtheorem{remark}{Remark}
\newcommand{\E}{\mathbb{E}}\newcommand{\R}{\mathbb{R}}\newcommand{\N}{\mathcal{N}}
\newcommand{\Pf}{\mathbb{P}}
\newcommand{\Sig}{\Sigma^{\star}}
\newcommand{\D}{\mathcal{D}}
\title{When Is a Graph a Covariance?\\[2pt] Bayesian Residual Propagation under Covariance Uncertainty}
\author{Robert Richardson\\
Department of Statistics, Brigham Young University\\
\texttt{richardson@stat.byu.edu}}
\date{}

\begin{document}
\maketitle
\begin{abstract}
Observing a prediction error at one node can help correct predictions elsewhere, but the benefit depends on the residual dependence between nodes. A graph suggests where that dependence might occur, yet does not establish its sign or strength. We develop a Bayesian model of residual covariance to determine how revealed errors should update a fixed feature-based predictor. For a fixed set of revealed nodes, we derive an exact identity for the change in expected squared error under linear residual propagation. The identity characterizes the optimal linear update and expresses the excess error of any other update as an exact quadratic. Under covariance uncertainty, the Bayes-optimal fixed linear update uses the posterior mean covariance. We average over four positive-semidefinite covariance families representing no dependence, positive dependence, negative dependence, and dependence whose sign can alternate with graph distance. Across eight graph benchmarks, four show posterior support within the distance-profile family for negative covariance between neighbors and alternating signs with distance. Among the evaluated methods, ours achieves the smallest worst-case recall gap to the best method on each graph. The method improves squared error over residual propagation on five of eight graphs, but performs substantially worse on one dense heterophilous graph. In these experiments, selecting the highest-posterior family performs similarly to model averaging.

\end{abstract}

\section{Introduction}

We consider prediction on a graph when a feature-based predictor is available and only some node labels can be observed. For each query, the task is to predict which nodes are relevant. Each observed label reveals a prediction error that may help correct predictions at other nodes. An error at one node may call for an adjustment in the same direction elsewhere, an adjustment in the opposite direction, or no adjustment. The graph alone does not tell us which response is appropriate or how large the adjustment should be. We use the graph to construct candidate models of residual covariance, allowing positive, negative, and sign-changing dependence. We learn a posterior over these models from previously labelled queries. For a new query, the resulting covariance estimate guides both the selection of nodes for labelling and the updates to predictions after their labels are observed.

\paragraph{Contributions.}  First, for centred residuals and a fixed set of revealed nodes, we derive an exact identity for the reduction in expected squared error under linear propagation. The identity gives the optimal linear update and expresses the excess error of any other update as an exact quadratic. For a single revealed node, any nonzero adjustment increases expected error when the true residual covariance with the target node is zero. A positive propagation coefficient also increases expected squared error when that covariance is
negative.

Second, we derive a Bayesian rule for propagation when the residual covariance is unknown. Under expected squared-error loss, the optimal linear rule with coefficients fixed before observing the new labels uses the posterior mean residual covariance. For a single reveal, the adjustment coefficient is zero exactly when the posterior mean cross-covariance is zero, so no separate threshold is needed to turn propagation off.

Third, we construct four positive-semidefinite covariance families representing no residual dependence, positive dependence, negative dependence, and dependence whose sign can alternate with graph distance. Proposition 1 shows that the standard Laplacian covariance model considered here has nonnegative covariance entries and nonnegative propagation coefficients. It therefore cannot represent negative residual covariance or covariance that alternates in sign with graph distance. We evaluate the resulting propagation decisions on eight graph benchmarks with varying levels of homophily.

\subsection{Relation to prior work}

Our work builds on residual propagation and methods for learning on heterophilous graphs.

\paragraph{Residual propagation.} Previous work has shown that residuals from feature-based predictors can retain graph structure and that propagating these residuals can improve predictions. \citet{jia2020} model GNN regression residuals with a graph-structured Gaussian distribution and estimate a signed dependence parameter by marginal likelihood. Their full model therefore allows both positive and negative residual correlation. Their lightweight LP-GNN variant and the correction step of \citet{huang2021} assume positive residual dependence. Our experiments include Laplacian propagation and the correction step of \citet{huang2021} as baselines. They do not include Jia and Benson's learned signed model. Our method instead places a posterior over several covariance families, including a null model and a distance-based family that permits uniformly negative or alternating covariance across positive graph distances. We hold the feature-based predictor fixed and use the inferred covariance to update its predictions.

\paragraph{Heterophily and selective propagation.} Methods that encourage neighbouring nodes to agree can perform poorly on heterophilous graphs, where connected nodes often have different labels \citep{pei2020,zhu2020,lim2021}. \citet{choi2025} address uncertainty about edges with an adaptive method that can block propagation. Other methods allow propagation to adapt to the graph: FAGCN combines low- and high-frequency components using weights that can be negative \citep{bo2021fagcn}, while GPR-GNN learns signed polynomial coefficients \citep{chien2021gprgnn}. These two methods learn their propagation weights as part of training the predictor. We hold the feature-based predictor fixed and use residuals from previously labelled queries to infer which covariance structures are supported by the data. The resulting posterior determines how newly observed errors update the remaining predictions. Label heterophily alone does not determine the residual dependence that remains after fitting the predictor.

\paragraph{Our distinction.} We focus on uncertainty about graph-based residual covariance and its consequences for propagation. The update has the standard Bayes-linear form \citep{goldstein2007}. Our risk identity makes explicit how the residual covariance determines its benefit and the cost of using a different linear update. \citet{xu2010} interpret graph smoothing as kriging and estimate graph-based correlations from observed responses. \citet{borovitskiy2021} develop Laplacian-based covariance models on graphs. Learning kernels from data also has a substantial literature, including kernel alignment and multiple kernel learning \citep{cristianini2001,gonen2011}. We target the covariance of residuals from a cross-fitted predictor \citep{chernozhukov2018} and place a posterior over covariance families that allow different signs and distance patterns.

\section{The Propagation Identity}
\label{sec:theory}

\paragraph{Setting.}
For a given query, a pool of $n$ items has unknown binary relevance
$g\in\{0,1\}^n$. A fitted predictor provides scores $m\in[0,1]^n$,
and the residual vector $r=g-m$ records its prediction errors.
We observe the relevance of $B$ items and use their residuals to
update predictions for the remaining items. We evaluate both
prediction error and ranking quality. The covariance we seek to
estimate is
\begin{equation}
   \Sig
   = \operatorname{Cov}(r)
   = \E[rr^\top]-\E[r]\E[r]^\top .
   \label{eq:sigdef}
\end{equation}

Cross-fitting prevents the predictor from being fitted on the
labels used to evaluate it, but does not guarantee zero-mean
residuals. In our data, pooled residual means reach $0.031$.
On one graph, centring removes $60\%$ of the raw edge second
moment. Much of the apparent agreement between neighbouring
residuals can therefore come from prediction bias. We estimate
a pooled bias from training queries and adjust the scores for
each new query without using its labels. This correction does
not guarantee that every residual component has mean zero.
The theorem below assumes centred residuals. Without that assumption the same algebra applies with the
residual second moment $\E[rr^\top]$ in place of the covariance $\Sig$, though an affine rule would then do strictly
better, which is why we correct the bias first and keep the linear form.

\begin{theorem}[Propagation identity, multiple reveals]
\label{thm:delta}
Let $r$ satisfy $\E[r]=0$, with covariance
$\Sig=\E[rr^\top]$. Let $S$ be a set of nodes fixed before
observing the residuals, and let $U$ be its complement.
Suppose the observed residuals are
$z_S=r_S+\varepsilon$, where $\varepsilon$ is independent
of $r$, $\E[\varepsilon]=0$, and
$\operatorname{Cov}(\varepsilon)=\sigma^2I$.
Write $V_S=\Sig_{SS}+\sigma^2I\succ0$.

For any fixed matrix $C$, consider the linear update
$\hat r_U=Cz_S$. Its reduction in total expected squared
prediction error is
\begin{equation}
\begin{split}
   \Delta(C;S)
   &:= \E\bigl[\|r_U\|_2^2\bigr]
       -\E\bigl[\|r_U-Cz_S\|_2^2\bigr]\\
   &=2\operatorname{tr}\!\bigl(C\Sig_{SU}\bigr)
     -\operatorname{tr}\!\bigl(CV_SC^\top\bigr).
\end{split}
   \label{eq:delta}
\end{equation}
The unique maximiser is $C^\star=\Sig_{US}V_S^{-1}$.
For every $C$,
\begin{equation}
\begin{split}
   \Delta(C^\star;S)-\Delta(C;S)
   &=\operatorname{tr}\!\bigl[
      (C-C^\star)V_S(C-C^\star)^\top
      \bigr]\\
   &=\bigl\|(C-C^\star)V_S^{1/2}\bigr\|_F^2 .
\end{split}
   \label{eq:quad}
\end{equation}
\end{theorem}

The proof expands the squared errors and completes the square.
It is given in the supplement. Equation~\eqref{eq:quad} gives
the exact excess error from using $C$ instead of the optimal
linear update. Under a joint Gaussian model for the residuals
and observation noise, $C^\star z_S$ is the posterior mean of
$r_U$. The identity itself does not require Gaussianity.
It shows which linear update to use when the covariance is
known. The statistical task is to estimate that covariance.

\begin{corollary}[Single reveal]
\label{cor:scalar}
For a single revealed node $a$ and a single target node $i$,
write $c=C_{ia}$ and $V=\Sig_{aa}+\sigma^2>0$.
The reduction in expected squared error is
\[
   \Delta_i(c)=2c\Sig_{ia}-Vc^2.
\]
It is maximised at $c^\star=\Sig_{ia}/V$, and
\[
   \Delta_i(c^\star)-\Delta_i(c)
   =V(c-c^\star)^2.
\]
\end{corollary}

\begin{corollary}[Propagation with zero residual covariance]
\label{cor:hurt}
For a single revealed node $a$ and target node $i$, suppose
$\Sig_{ia}=0$ and $V=\Sig_{aa}+\sigma^2>0$.
Then every nonzero propagation coefficient $c$ gives
\[
   \Delta_i(c)=-Vc^2<0.
\]
Thus, any nonzero linear adjustment increases expected squared
error when the true residual cross-covariance is zero.
\end{corollary}

If $\Sig_{ia}<0$ and $c>0$, both terms in
$\Delta_i(c)=2c\Sig_{ia}-Vc^2$ are negative, so any positive
propagation coefficient increases expected squared error. The graph does not appear in \eqref{eq:delta}. It supplies
candidate models for the sign, magnitude, and pattern of
residual covariance. Corollary~\ref{cor:hurt} describes one
cost of an incorrect covariance assumption. If the model
induces a nonzero adjustment between nodes whose residual
covariance is zero, that adjustment increases expected error.

\paragraph{Fixed and adaptive reveals.}
Theorem~\ref{thm:delta} assumes that $S$ is fixed in advance.
Its derivation uses moments such as
$\E[z_Sr_U^\top]=\Sig_{SU}$. When previously observed outcomes
determine which nodes are revealed next, the selected set
depends on the residuals. The unconditional covariance blocks
need not describe the moments under that selection.
Our experiments select nodes adaptively
(\S\ref{sec:exp}), so the theorem does not directly establish
optimality for the experimental procedure.

\paragraph{Squared error and ranking.}
The identity concerns expected squared error. It does not
directly establish an improvement in ranking or recall.
The supplement bounds pairwise disagreement with the oracle
linear ranking using the propagation regret
$\mathcal R(C)=\Delta(C^\star;S)-\Delta(C;S)$
and the oracle score margins. It gives both Gaussian and
distribution-free bounds. The bound is loose in simulation,
at $5$ to $12$ times the realised misordering count, and
provides no recall guarantee.

\paragraph{Observation noise and regularisation.}
In Theorem~\ref{thm:delta}, $\sigma^2$ is the variance of
observation noise. Our experiments reveal exact labels, so
their observation-noise variance is zero. We nevertheless
add a nugget of $0.05$ when computing the update to regularise
the matrix inversion and moderate the adjustment.
For a fixed reveal set and centred residuals, the identity
can evaluate this regularised coefficient matrix using the
true noise variance of zero. It does not establish that
the regularised update or the chosen nugget is optimal.

\section{Limits of Laplacian Covariance Models}
\label{sec:gmrf}

\begin{proposition}[Sign rigidity]
\label{prop:rigid}
Let $L$ be the combinatorial Laplacian of an undirected graph
with nonnegative edge weights. For $\lambda>0$, define
$K=(I+\lambda L)^{-1}$. Let $S$ be the revealed set and
$U$ its complement. The linear update based on $K$ for
exact observations has coefficient matrix
\[
   C(K)=K_{US}K_{SS}^{-1}.
\]
Every entry $c_{ij}$ of $C(K)$ is nonnegative.

For a single reveal $S=\{j\}$ and a target node $i$, if
$c_{ij}>0$ and $\Sig_{ij}<0$, then
Corollary~\ref{cor:scalar} gives $\Delta_i(c_{ij})<0$.
Thus, for every $\lambda>0$, a nonzero update of this form
increases expected squared error when the true residual
covariance between the two nodes is negative.
\end{proposition}

For a single reveal, a zero coefficient leaves the prediction
unchanged and gives zero gain. With several reveals, the gain
depends on all the coefficients and the relevant covariance
blocks. One negative cross-covariance alone does not determine
whether the combined update helps or hurts.

\paragraph{Why the coefficients are nonnegative.}
Nonnegative entries in $K$ alone are not enough to establish
Proposition~\ref{prop:rigid}. A positive-definite covariance
matrix with nonnegative entries can produce negative
coefficients after multiplication by $K_{SS}^{-1}$.
The result follows from the particular structure of the
Laplacian precision matrix.

Write $A=K^{-1}=I+\lambda L$. This matrix has nonpositive
off-diagonal entries and is strictly diagonally dominant
with positive diagonal entries. It is therefore a nonsingular
$M$-matrix. Its principal submatrix $A_{UU}$ has an
entrywise nonnegative inverse. The block identity
\[
   C(K)=K_{US}K_{SS}^{-1}
       =-\bigl(A_{UU}\bigr)^{-1}A_{US}
\]
then gives the result, because $A_{US}$ is entrywise
nonpositive.

The result also holds when the update includes a positive
nugget $\tau>0$. Let $E_S$ be the matrix whose columns are
the coordinate vectors for the nodes in $S$, so that
$E_S^\top r=r_S$. The update matrix for all nodes is
\[
   K E_S(K_{SS}+\tau I)^{-1}
   =
   \bigl(A+\tau^{-1}E_SE_S^\top\bigr)^{-1}
   \tau^{-1}E_S.
\]
Adding the nonnegative diagonal matrix
$\tau^{-1}E_SE_S^\top$ preserves the nonsingular
$M$-matrix property. The inverse and $E_S$ are both
entrywise nonnegative, so the update matrix is also
entrywise nonnegative. Its rows indexed by $U$ give
the coefficients for the unrevealed nodes.
This argument applies to the Laplacian covariance model.
It does not automatically apply to the other covariance
families considered below.

\paragraph{Covariance patterns excluded by the model.}
Since $A$ is a nonsingular $M$-matrix, its inverse $K$ is
also entrywise nonnegative. The model therefore cannot
represent negative residual covariance or covariance that
alternates in sign with graph distance. More generally,
neither pattern can be represented by an entrywise
nonnegative covariance kernel. Section~\ref{sec:exp}
examines evidence for these patterns in the graph benchmarks.

This restriction concerns signs, not monotonicity.
On an irregular graph, the entries of
$(I+\lambda L)^{-1}$ need not decrease monotonically with
graph distance. The argument does not require them to.
\section{Bayesian Covariance Inference and Propagation}
\label{sec:method}

\subsection{The Bayes-optimal linear propagation strength}

Let $\D$ denote data from previously labelled queries.
We use these data to choose a propagation coefficient for a
new query. Under expected squared-error loss, this decision
depends on the posterior mean covariance. Posterior dispersion
alone provides no additional reason to reduce propagation.

\begin{remark}[Expected gain depends on the posterior mean]
\label{rem:linear}
For a fixed coefficient $c$, the single-reveal gain satisfies
\[
   \E[\Delta_i(c)\mid\D]
   =2c\,\E[\Sig_{ia}\mid\D]
    -c^2\bigl(\E[\Sig_{aa}\mid\D]+\sigma^2\bigr).
\]
The same property holds for a fixed matrix $C$ in
\eqref{eq:delta}. Holding the posterior mean covariance
fixed, changing posterior dispersion does not change the
expected gain. Uncertainty affects this decision through
the posterior mean, without a separate dispersion penalty.
\end{remark}

Maximising this expected gain gives the following rule.

\begin{proposition}[Bayes-optimal linear propagation]
\label{prop:cb}
Conditional on $\D$ and $\Sig$, suppose the new residuals
have mean zero and covariance $\Sig$, with the independent
observation noise of Theorem~\ref{thm:delta}.
Assume finite posterior mean covariance and a reveal set
fixed before observing the new residuals.

For a single reveal, write
$\bar V=\E[\Sig_{aa}\mid\D]+\sigma^2>0$.
The coefficient maximising posterior expected gain is
\begin{equation}
   c^B_{ia}
   =\frac{\E[\Sig_{ia}\mid\D]}
          {\E[\Sig_{aa}\mid\D]+\sigma^2},
   \label{eq:cb}
\end{equation}
and, for every coefficient $c$ chosen before observing $z_a$,
\[
   \E[\Delta_i(c^B_{ia})-\Delta_i(c)\mid\D]
   =\bar V(c-c^B_{ia})^2.
\]
For multiple reveals, let
$\bar V_S=\E[\Sig_{SS}\mid\D]+\sigma^2I\succ0$.
Then
\[
   C^B=\E[\Sig_{US}\mid\D]\,\bar V_S^{-1},
\]
with
\[
   \E[\Delta(C^B;S)-\Delta(C;S)\mid\D]
   =\bigl\|(C-C^B)\bar V_S^{1/2}\bigr\|_F^2
\]
for every coefficient matrix $C$ chosen before observing $z_S$.
\end{proposition}

Both results follow by taking posterior expectations in
\eqref{eq:delta} and completing the square. For a single
reveal, the optimal coefficient is unique and equals zero
exactly when the posterior mean cross-covariance is zero.
Choosing not to propagate therefore requires no separate
threshold.

\begin{remark}[Scope of the optimality]
\label{rem:linscope}
The optimum is over rules linear in $z_S$, with coefficients
fixed before $z_S$ is observed. A full posterior predictive
calculation can also update model probabilities and parameter
distributions using $z_S$, producing a prediction that need
not be linear. Proposition~\ref{prop:cb} does not establish
optimality over that larger class. Updating family weights within a query changes them by at most $0.01$ over $B=20$ reveals
(supplement).
\end{remark}

\subsection{A posterior over covariance shapes}

We use the graph to construct four covariance families:
zero cross-covariance, positive dependence, negative
dependence, and a geometric distance profile.
Let $R_\lambda=(I+\lambda L)^{-1}$, with $L$ the
combinatorial Laplacian and $\lambda>0$. Define
$K^G(\lambda)=D_\lambda^{-1/2}R_\lambda D_\lambda^{-1/2}$,
where $D_\lambda$ is the diagonal matrix containing the
diagonal entries of $R_\lambda$. Thus, $K^G$ has unit
diagonal. With $a>0$ the common residual variance and
$H_h$ the matrix indicating pairs at graph distance $h$,
the families are
\[
\begin{array}{ll}
M_0: & \Sigma=aI,\\
M_+: & \Sigma=aI+b\bigl(K^G(\lambda)-I\bigr),\quad b>0,\\
M_-: & \Sigma=aI+b\bigl(K^G(\lambda)-I\bigr),\quad b<0,\\
M_H: & \Sigma=aI+\theta\displaystyle\sum_{h\ge1}
                     \rho^{h-1}H_h,\\
     & \qquad \theta\in\R,\quad \rho\in(-1,1).
\end{array}
\]
Every candidate has diagonal entries equal to $a$.
The comparison therefore concerns cross-node covariance
while holding marginal variance fixed.

$M_0$ gives no linear propagation between distinct nodes.
$M_+$ models the positive residual dependence assumed by
LP-GNN and the correction step of \citet{huang2021}.
$M_-$ reverses the off-diagonal sign.
For $M_H$, covariance at distance $h$ is
$\theta\rho^{h-1}$. When $\theta<0$ and $\rho>0$, it is
negative at every finite positive distance. When
$\theta\neq0$ and $\rho<0$, its sign alternates with distance.
Neither signed pattern can be represented by an entrywise
nonnegative covariance kernel. Each non-null family has
two parameters beyond $a$, one scale and one range.

\paragraph{Valid covariance matrices.}
The displayed forms need not be positive semidefinite
without parameter restrictions. In our experiments, a free
distance profile has negative minimum eigenvalues on two
graphs. An indefinite matrix cannot define the Gaussian
likelihood or predictive variances used for acquisition.
Writing $\eta_m$ for the parameters of family $m$, and taking $\epsilon=10^{-3}a$, we restrict the
parameter space to
\begin{equation}
   \Theta_m
   =\bigl\{\eta_m:
      \Sigma_m(\eta_m)\succeq\epsilon I\bigr\},
   \label{eq:psd}
\end{equation}
subject to the sign and range restrictions above, with
$0<\epsilon<a$. For fixed $\lambda$, the admissible interval
for $b$ is determined by the extreme eigenvalues of
$K^G(\lambda)-I$. For fixed $\rho$, the interval for $\theta$ is determined by the extreme eigenvalues of
$\sum_{h\ge1}\rho^{h-1}H_h$.
Posterior averaging preserves this covariance constraint.

\paragraph{Marginal likelihoods.}
For each non-null family, we use a $\N(0,(a/2)^2)$ prior
on its scale parameter, truncated and renormalised over
the admissible interval for each range value and family.
A narrow interval is therefore not penalised simply for
containing less of the untruncated prior.
We integrate the scale by quadrature and the range
parameter ($\lambda$ or $\rho$) over a grid.

The Gaussian likelihood is a working model for comparing
the second-moment structure of binary residuals, not a
claim that they are Gaussian. It treats residual vectors
from different target classes as independent, so the
likelihood is composite. Numerical integration does not
correct for this dependence. Section~\ref{sec:exp} examines
sensitivity to the likelihood specification.

Under a uniform model prior, the integrated working
likelihoods give weights $\Pf(M_m\mid\D)$, and
\begin{equation}
   \E[\Sig\mid\D]
   =\sum_m\Pf(M_m\mid\D)\,
      \E[\Sig\mid\D,M_m].
   \label{eq:bma}
\end{equation}
Within each family, the scale is averaged using the same
quadrature as the evidence. Range-grid points are weighted
by their integrated likelihoods times their prior
probabilities. Equation~\eqref{eq:bma} therefore accounts
for parameter uncertainty within families and uncertainty
across families. The common diagonal $a$ is the pooled variance of the cross-fitted training residuals, plugged in rather
than given a prior, which makes the method empirical Bayes in that one parameter. Binary residuals are heteroscedastic, with variance $m_i(1-m_i)$ under a calibrated mean, so a common
diagonal is an approximation that leaves the families distinguished by off-diagonal structure alone.
We average the covariance matrices before forming the
update in Proposition~\ref{prop:cb}, rather than averaging
model-specific updates.

\section{Experiments}
\label{sec:exp}

\paragraph{Three questions.}
We ask whether residual covariance exhibits patterns excluded
by positive graph smoothing (Q1), how reliably the posterior
distinguishes covariance families and how much that choice
matters for prediction (Q2), and whether the resulting decisions
improve prediction and ranking (Q3).

\paragraph{Data.}
We use eight graphs: the citation networks \textsc{cora} and
\textsc{citeseer} \citep{sen2008}, the WebKB/actor graphs
\textsc{texas}, \textsc{cornell}, \textsc{wisconsin} and
\textsc{film} \citep{pei2020}, and the Wikipedia article
networks \textsc{chameleon} and \textsc{squirrel}
\citep{rozemberczki2021}. Graph sizes range from $183$ to
$7600$ nodes, mean degrees from $2.74$ to $76.27$, and edge
homophily from $0.061$ to $0.810$. For the Wikipedia graphs,
we discretise the traffic target into five quantile classes.

The Wikipedia graphs have a known source of leakage across
splits: nodes can share both a neighbourhood and a target
\citep{platonov2023}. In our copies, this affects $36\%$ of
\textsc{squirrel} nodes and $46\%$ of \textsc{chameleon}
nodes. Their results should be interpreted with this
limitation in mind.

\paragraph{Protocol.}
Each target class defines a query. The initial predictor is
a cross-fitted logistic regression on node features, and
covariance estimation leaves out the target class.
Pools are connected neighbourhoods of up to $200$ nodes. Every covariance operation is therefore on a matrix of that size, at cost $O(200^3)$, and graph distances
are truncated at $h\le4$.
We use connected pools because uniform node sampling can
remove most graph connections. On \textsc{cora}, uniform
samples of this size have mean degree $0.29$, compared
with $3.90$ in the full graph.

Methods share the initial predictions, labelled set and
budget of $B=20$ additional reveals. Each method selects
nodes using its own covariance, except that \textsc{c\&sp}
uses the nodes selected by \textsc{gmrf}. Methods are
evaluated on their remaining unrevealed nodes, so the
evaluation sets can differ.

Acquisition uses an upper confidence bound (UCB), with
an exploration bonus based on
\[
   \sqrt{\operatorname{diag}\!\left(
   K-K_{:S}(K_{SS}+\tau I)^{-1}K_{S:}
   \right)},
\]
where $\tau=0.05$ is the regularisation nugget.
The covariance therefore affects both node selection
and prediction updates.

We report mean squared error on unrevealed nodes and
recall@$k$, where $k$ is the number of relevant items
in the pool. This count is unavailable at deployment,
so recall@$k$ evaluates ranking with an oracle-supplied
cutoff. Paired bootstrap intervals resample test pools
within each graph. On the three small WebKB graphs,
a pool covers most of the giant component. These intervals
describe within-graph variability, not replication across
independent graphs.

\paragraph{Baselines.}
We compare a feature kernel (\textsc{feat}), Laplacian
smoothing (\textsc{gmrf}), and residual propagation based
on \citet{huang2021}, with autoscale (\textsc{c\&s}) or
without it (\textsc{c\&sp}). The plain correction step
performs better in our setting with only $20$ reveals
per pool, where autoscale amplifies the corrections.
We therefore report \textsc{c\&sp} as the primary
residual-propagation baseline.

\subsection{Residual covariance differs across graph regimes}
\label{sec:rangepred}

\textbf{Q1: the fitted profiles support patterns excluded
by positive smoothing.}
On \textsc{texas}, \textsc{cornell}, \textsc{wisconsin}
and \textsc{chameleon}, the posterior probability of
negative covariance at distance one is at least $0.97$
conditional on $M_H$. Under the same conditioning,
the probability that $\rho<0$, giving alternating signs
with distance, is at least $0.93$.
An entrywise nonnegative covariance cannot represent
these patterns. Proposition~\ref{prop:rigid} also shows
that the standard Laplacian update has nonnegative
propagation coefficients.

In contrast, \textsc{cora} has residual covariance
$+0.0122$ on citation edges, compared with $+0.0000$
on matched non-edges. Edge covariance is positive
for all seven classes.

\paragraph{Posterior shape evidence.}
Table~\ref{tab:weights} reports model weights and
the number of near-disjoint training pools.
Only one such pool is available on each small WebKB
graph because a neighbourhood of the requested size
covers the giant component. Pool counts indicate
the limited replication behind these weights.

The posterior favours $M_+$ on \textsc{cora} and
\textsc{squirrel}, and $M_H$ on \textsc{citeseer},
the three WebKB graphs and \textsc{chameleon}.
On \textsc{film}, weights of $0.54$ on $M_+$ and
$0.33$ on $M_0$ indicate less decisive support.

\begin{table}[t]
\centering
\caption{Posterior over covariance shape.
$M_0$: zero cross-covariance.
$M_+$: positive Laplacian-based covariance.
$M_-$: negative Laplacian-based covariance.
$M_H$: signed geometric profile.
The last column counts near-disjoint training pools.}
\label{tab:weights}
\begin{tabular}{lrrrrr}
\toprule
graph & $M_0$ & $M_+$ & $M_-$ & $M_H$ & pools \\
\midrule
\textsc{cora}      & $0.00$ & $\bm{1.00}$ & $0.00$ & $0.00$ & $12$ \\
\textsc{citeseer}  & $0.00$ & $0.08$ & $0.00$ & $\bm{0.92}$ & $12$ \\
\textsc{texas}     & $0.06$ & $0.01$ & $0.03$ & $\bm{0.91}$ & $1$ \\
\textsc{cornell}   & $0.05$ & $0.00$ & $0.18$ & $\bm{0.78}$ & $1$ \\
\textsc{wisconsin} & $0.07$ & $0.01$ & $0.03$ & $\bm{0.89}$ & $1$ \\
\textsc{film}      & $0.33$ & $\bm{0.54}$ & $0.03$ & $0.11$ & $12$ \\
\textsc{chameleon} & $0.00$ & $0.10$ & $0.00$ & $\bm{0.90}$ & $12$ \\
\textsc{squirrel}  & $0.00$ & $\bm{1.00}$ & $0.00$ & $0.00$ & $12$ \\
\bottomrule
\end{tabular}
\end{table}

\subsection{Family identification and decision stability}
\label{sec:weights}

\textbf{Q2: family selection is sensitive to the likelihood,
while selection and averaging give similar predictions
under the original likelihood.}

\paragraph{Likelihood sensitivity.}
The likelihood treats residual vectors from different target classes as independent, making it a composite
likelihood \citep{ribatet2012}. Unadjusted composite likelihoods overcount dependent evidence and concentrate the posterior, which pushes
averaging toward the highest-posterior family. An adjustment would need an effective sample size we cannot
justify for overlapping pools. Recomputing the evidence
with one class per pool, cycling over three choices,
changes the selected family on six of eight graphs.
Only \textsc{cora} and \textsc{squirrel} retain the
same selection. This change reduces both the number
of likelihood terms and their dependence, so it does
not isolate the effect of dependence alone.
Under the reduced likelihoods the averaged covariance is worse on six of eight graphs, by $0.0008$ in
squared error on average. \textsc{cora} and \textsc{squirrel} hold $M_+$ under every class choice; the
graphs selecting $M_H$ scatter.

Under the original likelihood, selecting the
highest-posterior family (\textsc{map}) and averaging
over families (\textsc{bayes}) differ by at most
$0.0002$ in squared error and $0.0039$ in recall.
The largest recall difference occurs on \textsc{citeseer}.
Similar predictive performance does not imply similar
covariance estimates or propagation coefficients.

\paragraph{Recovery of the generating family.}
We simulate residuals on the \textsc{cora} graph under
each covariance family. Gaussian simulations assess
selection under a correctly specified model.
Thresholding a latent Gaussian produces binary data
for a misspecification check. In the binary case,
the generating-family label refers to the latent
covariance, not necessarily the observed covariance.

The recovery table in the supplement reports selection frequencies.
The selected family matches the generating label in
$73\%$ of Gaussian and $77\%$ of binary simulations.
$M_+$ is selected in every replicate generated from
that family. $M_-$ is recovered in every Gaussian
replicate but is selected as $M_0$ in half of the
binary replicates. When $M_0$ generates the data,
it is selected in $92\%$ and $83\%$ of replicates,
respectively.

In the Gaussian simulations, $M_H$ is always selected
as $M_+$. A geometric profile dominated by positive
covariance can resemble the positive Laplacian family.
These results show that the posterior weights do not
reliably distinguish all generating families.

\paragraph{Prediction cost of family selection.}
The supplement evaluates simulated propagation rules
using the oracle regret in Theorem~\ref{thm:delta},
normalised by the cost of not propagating.
Under Gaussian residuals, supplying the generating
family reduces normalised regret by $0.008$ relative
to averaging, while always using $M_+$ increases it
by $0.323$. Thus, in these simulations, restricting
the available covariance patterns matters more for
prediction than knowing the generating family.
The binary experiment uses a latent-scale oracle,
so its result is a discrepancy diagnostic rather
than a fraction of attainable predictive gain.

\subsection{Prediction and ranking performance}
\label{sec:results}

\textbf{Q3: recall improves over the primary
residual-propagation baseline on every graph.
Squared-error improvements are not uniform.}

\paragraph{Results against baselines.}
Table~\ref{tab:main} reports both metrics.
Compared with \textsc{c\&sp}, \textsc{bayes} has lower
squared error on five graphs and higher error on
\textsc{cora}, \textsc{citeseer} and \textsc{squirrel}.
It has lower squared error than \textsc{gmrf} on
all eight graphs.

\begin{table*}[!t]
\centering\scriptsize
\caption{Mean squared error on unrevealed items
(lower is better) and recall@$k$ (higher is better).
All methods use $B=20$, leave-class-out covariance
estimation and PSD-constrained covariances.
\textsc{c\&sp} is the plain correction step of
\citet{huang2021}. \textsc{prof} uses the $M_H$ family alone,
\textsc{map} uses the single highest-posterior family,
and \textsc{bayes} uses the model average.
Leading zeros are omitted.}
\label{tab:main}
\begin{tabular}{lrrrrrr@{\hskip 1.0em}rrrrrr}
\toprule
& \multicolumn{6}{c}{squared error}
& \multicolumn{6}{c}{recall@$k$}\\
\cmidrule(r){2-7}\cmidrule(l){8-13}
graph
& \textsc{feat} & \textsc{gmrf} & \textsc{c\&sp}
& \textsc{prof} & \textsc{map} & \textsc{bayes}
& \textsc{feat} & \textsc{gmrf} & \textsc{c\&sp}
& \textsc{prof} & \textsc{map} & \textsc{bayes}\\
\midrule
\textsc{cora}
& .0625 & .0487 & .0473 & .0480 & .0477 & .0477
& .7357 & .7805 & .7066 & .7495 & .7586 & .7586\\
\textsc{citeseer}
& .0784 & .0618 & .0574 & .0583 & .0583 & .0583
& .5892 & .5988 & .5501 & .5763 & .5763 & .5802\\
\textsc{texas}
& .0688 & .0570 & .0541 & .0536 & .0536 & .0536
& .8255 & .7977 & .7977 & .8255 & .8255 & .8255\\
\textsc{cornell}
& .0617 & .0571 & .0561 & .0529 & .0529 & .0530
& .7482 & .7423 & .7532 & .7959 & .7959 & .7959\\
\textsc{wisconsin}
& .0455 & .0358 & .0345 & .0315 & .0315 & .0315
& .8165 & .8096 & .7979 & .8705 & .8705 & .8705\\
\textsc{film}
& .2206 & .1400 & .1366 & .1364 & .1363 & .1363
& .3854 & .4393 & .3825 & .4470 & .4479 & .4475\\
\textsc{chameleon}
& .1030 & .0944 & .0809 & .0797 & .0797 & .0799
& .6554 & .6626 & .6323 & .6653 & .6653 & .6645\\
\textsc{squirrel}
& .2322 & .1965 & .1363 & .1402 & .1488 & .1488
& .4349 & .4503 & .3934 & .4545 & .4473 & .4473\\
\bottomrule
\end{tabular}
\end{table*}

Table~\ref{tab:ci} gives paired bootstrap intervals
against \textsc{c\&sp}. Four of the five squared-error
improvements have intervals excluding zero.
The exception is \textsc{film}, whose improvement
is in the point estimate only.
The losses on \textsc{citeseer} and \textsc{squirrel}
have intervals excluding zero. On \textsc{cora},
the lower endpoint rounds to zero, so we do not
claim a resolved difference.

Recall gains over \textsc{c\&sp} range from $0.028$
to $0.073$, with all eight intervals excluding zero.
This advantage does not extend to every method:
\textsc{gmrf} has higher recall on both citation graphs.
The squared-error identity does not provide a
ranking guarantee.

\begin{table*}[!t]
\centering\footnotesize
\caption{Paired bootstrap $95\%$ intervals for
\textsc{bayes} minus the primary residual-propagation
baseline, \textsc{c\&sp}, resampling test pools within
each graph. Negative differences favour \textsc{bayes}
for squared error, and positive differences favour
\textsc{bayes} for recall.}
\label{tab:ci}
\begin{tabular}{lcc}
\toprule
graph & squared error & recall@$k$\\
\midrule
\textsc{cora}
& $+0.0004\,[+0.0000,+0.0008]$
& $+0.0520\,[+0.0338,+0.0718]$\\
\textsc{citeseer}
& $+0.0009\,[+0.0001,+0.0016]$
& $+0.0301\,[+0.0082,+0.0522]$\\
\textsc{texas}
& $-0.0005\,[-0.0009,-0.0001]$
& $+0.0278\,[+0.0189,+0.0378]$\\
\textsc{cornell}
& $-0.0031\,[-0.0041,-0.0023]$
& $+0.0427\,[+0.0349,+0.0505]$\\
\textsc{wisconsin}
& $-0.0030\,[-0.0037,-0.0024]$
& $+0.0727\,[+0.0577,+0.0891]$\\
\textsc{film}
& $-0.0002\,[-0.0007,+0.0002]$
& $+0.0650\,[+0.0576,+0.0727]$\\
\textsc{chameleon}
& $-0.0010\,[-0.0016,-0.0005]$
& $+0.0323\,[+0.0218,+0.0436]$\\
\textsc{squirrel}
& $+0.0125\,[+0.0089,+0.0163]$
& $+0.0539\,[+0.0439,+0.0643]$\\
\bottomrule
\end{tabular}
\end{table*}

\paragraph{Performance across graphs.}
To assess performance across datasets, we measure
each method's recall shortfall from the best method
on each graph. \textsc{bayes} has the smallest maximum recall shortfall, $0.0219$, followed by \textsc{map} at
$0.0225$.
This does not mean either method leads on every graph.
For squared error, the maximum excess over the
best method is $0.0125$ for \textsc{bayes}, compared
with $0.0032$ for \textsc{c\&sp}.

\paragraph{The failure on \textsc{squirrel}.}
The largest squared-error loss occurs on
\textsc{squirrel}, where \textsc{bayes} exceeds
\textsc{c\&sp} by $0.0125$, with paired interval
$[0.0089,0.0163]$.
The posterior assigns essentially all weight to $M_+$.
Observed edge covariance is positive, at $0.00060$
compared with $0.00005$ off edges, but that sign alone
does not establish that $M_+$ is the best predictive
covariance model. \textsc{prof}, the $M_H$-only method, has lower test squared error, $0.1402$ versus $0.1488$.
Here, marginal likelihood and test squared error
favour different families.

Family choice can also help. On \textsc{cora},
selecting $M_+$ raises recall from $0.7495$ for
$M_H$ alone to $0.7586$.
The matching performance of \textsc{map} and
\textsc{bayes} shows that this gain comes from
family choice, with little additional benefit
from averaging in these experiments.

\section{Negative Results and Ablations}
\label{sec:fails}

We examined two alternatives.

Updating covariance-family weights within a query changed
them by at most $0.01$ over $B=20$ reveals. In these
experiments, the new observations therefore had little
effect on the family weights, as noted in
Remark~\ref{rem:linscope}.

Finally, evaluating candidate updates with
\eqref{eq:delta} after substituting an estimated covariance
selects the update optimal for that estimate by construction.
This criterion selected the plug-in update on all eight
graphs. It does not independently assess the covariance
estimate, which motivates comparing covariance models
through their working likelihoods.

\section{Discussion}
\label{sec:disc}

A graph suggests how prediction errors may be related,
but does not establish their covariance. We use a posterior
over covariance families to determine how observed errors
should update predictions elsewhere. Under the stated
assumptions, the posterior mean covariance gives the
Bayes-optimal fixed linear update. For a single reveal,
the propagation coefficient is zero when the posterior
mean cross-covariance is zero, so no separate threshold
is needed to turn propagation off. Across eight graphs,
model averaging and selecting the highest-posterior family give similar results. Model averaging has the
smallest maximum recall shortfall among the methods evaluated. Squared error improves over
the residual-propagation baseline on five graphs, but
the largest loss exceeds the largest improvement.

Several limitations remain. The theory concerns squared
error and provides no recall guarantee. Maximum excess
squared error is $0.0125$ for our method, compared with
$0.0032$ for the residual-propagation baseline, with the
largest loss on \textsc{squirrel}. Family selection changes
on six graphs under the reduced likelihood, and predictive
performance under that likelihood has not been evaluated.
The simulations also show that distinct covariance
families can be difficult to distinguish. The theory
assumes fixed reveal sets and centred residuals, while
the experiments select adaptively and use a pooled bias
correction. Evaluation is limited to eight node-classification
benchmarks and four covariance families, with only one
near-disjoint training pool on each small WebKB graph.
Extending the analysis to adaptive selection and allowing
residual dependence to vary across a graph are natural
next steps.

\subsection*{AI use statement}

In this work, we used generative AI tools to implement and run the analysis scripts, to help revise
text, to check internal consistency between the manuscript and the numbers its scripts produce. We did not use
generative AI tools to generate the research ideas, and reviewing for other venues is not applicable to this
work. We have reviewed all AI-assisted work: every reported number
was checked against the script named for it, and each proposition was verified independently of the drafting
process. We take responsibility for the final content of this work, including text, claims, or artifacts
produced with the aid of generative AI.

\bibliographystyle{plainnat}
\bibliography{refs}

\clearpage
\setcounter{section}{0}
\setcounter{table}{0}
\setcounter{figure}{0}
\setcounter{equation}{0}
\setcounter{theorem}{0}
\setcounter{proposition}{0}
\setcounter{corollary}{0}
\setcounter{lemma}{0}
\renewcommand{\thesection}{S\arabic{section}}
\renewcommand{\thetable}{S\arabic{table}}
\renewcommand{\theequation}{S\arabic{equation}}
\renewcommand{\thetheorem}{S\arabic{theorem}}
\renewcommand{\theproposition}{S\arabic{proposition}}
\renewcommand{\thecorollary}{S\arabic{corollary}}
\renewcommand{\thelemma}{S\arabic{lemma}}
\section*{Supplementary Material}
\addcontentsline{toc}{section}{Supplementary Material}

\section{Proofs}

\subsection{The propagation identity (Theorem 1)}

\begin{theorem}[Propagation identity, multiple reveals]
Let $S$ be revealed and $U$ unrevealed, with $S$ and the coefficient matrix $C$ fixed before the
residuals are observed, and $z_S=r_S+\varepsilon$ with $\varepsilon\sim(0,\sigma^2I)$
independent of $r$, let $\E[r]=0$, and let $V_S=\Sig_{SS}+\sigma^2I\succ0$. For any
$C\in\R^{|U|\times|S|}$ and $\hat r_U=Cz_S$,
\[
   \Delta(C;S):=\E\|r_U\|^2-\E\|r_U-Cz_S\|^2
   =2\operatorname{tr}\!\bigl(C\Sig_{SU}\bigr)-\operatorname{tr}\!\bigl(CV_SC^\top\bigr),
\]
uniquely maximised at $C^\star=\Sig_{US}V_S^{-1}$, and for every $C$,
\[
   \Delta(C^\star;S)-\Delta(C;S)
   =\operatorname{tr}\!\bigl[(C-C^\star)V_S(C-C^\star)^\top\bigr]
   =\bigl\|(C-C^\star)V_S^{1/2}\bigr\|_F^2 .
\]
\end{theorem}

\begin{proof}
Expanding the square,
$\E\|r_U-Cz_S\|^2=\E\|r_U\|^2-2\,\E[r_U^\top Cz_S]+\E\|Cz_S\|^2$, so
$\Delta(C;S)=2\,\E[r_U^\top Cz_S]-\E\|Cz_S\|^2$. For the cross term, $r_U^\top Cz_S$ is a scalar, so it
equals its own trace and
\[
   \E[r_U^\top Cz_S]=\E\operatorname{tr}(Cz_Sr_U^\top)=\operatorname{tr}\bigl(C\,\E[z_Sr_U^\top]\bigr)
   =\operatorname{tr}(C\Sig_{SU}),
\]
using $\E[z_Sr_U^\top]=\E[(r_S+\varepsilon)r_U^\top]=\Sig_{SU}$ because $\varepsilon\perp r$ and
$\E\varepsilon=0$. For the quadratic term,
$\E\|Cz_S\|^2=\E\operatorname{tr}(Cz_Sz_S^\top C^\top)=\operatorname{tr}(CV_SC^\top)$ with
$V_S=\E[z_Sz_S^\top]=\Sig_{SS}+\sigma^2I$. This gives the identity.

Since $V_S\succ0$, $C\mapsto-\operatorname{tr}(CV_SC^\top)$ is strictly concave
and the stationary point is the unique maximiser. Differentiating,
$\partial\Delta/\partial C=2{\Sig_{SU}}^{\!\top}-2CV_S=2\Sig_{US}-2CV_S$, which vanishes iff
$C=\Sig_{US}V_S^{-1}=C^\star$.

For the last claim, symmetry of $V_S$ gives ${\Sig_{SU}}={\Sig_{US}}^{\!\top}=(C^\star V_S)^\top=V_SC^{\star\top}$,
so $\operatorname{tr}(C\Sig_{SU})=\operatorname{tr}(CV_SC^{\star\top})$ and
\[
   \Delta(C;S)=2\operatorname{tr}(CV_SC^{\star\top})-\operatorname{tr}(CV_SC^\top),
   \qquad
   \Delta(C^\star;S)=\operatorname{tr}(C^\star V_SC^{\star\top}).
\]
Subtracting,
$\Delta(C^\star;S)-\Delta(C;S)
=\operatorname{tr}(C^\star V_SC^{\star\top})-2\operatorname{tr}(CV_SC^{\star\top})+\operatorname{tr}(CV_SC^\top)
=\operatorname{tr}[(C-C^\star)V_S(C-C^\star)^\top]$,
which is $\|(C-C^\star)V_S^{1/2}\|_F^2$ since $V_S\succ0$.
\end{proof}

\begin{remark}[Why centring is assumed]
The proof uses $\E[z_Sr_U^\top]=\Sig_{SU}$, which requires $\E[r]=0$. This is also what makes
$\Sig=\E[rr^\top]$ a covariance rather than a second moment. Cross-fitting does not deliver it, since
it prevents the mean from being fitted on the labels it is scored against but does not make it the
conditional mean. In the experiments the bias is estimated on training queries and applied as a
recalibration to held-out ones, which uses no labels from the query being scored.
\end{remark}

\begin{corollary}[Single reveal]
With $S=\{a\}$, $U=\{i\}$, $c=C_{ia}$ and $V=\Sig_{aa}+\sigma^2$, the identity reads
$\Delta_i(c)=2c\Sig_{ia}-Vc^2$, maximised at $c^\star=\Sig_{ia}/V$, with
$\Delta_i(c^\star)-\Delta_i(c)=V(c-c^\star)^2$. If $\Sig_{ia}=0$ but $c\neq0$ then
$\Delta_i(c)=-Vc^2<0$.
\end{corollary}

\subsection{Sign rigidity (Proposition 1)}

\begin{proposition}
Let $L$ be the combinatorial Laplacian of an undirected graph with nonnegative edge weights, let
$\lambda>0$, put
$K=(I+\lambda L)^{-1}$, and let $c_{ij}$ denote the entries of $C(K)=K_{US}K_{SS}^{-1}$. Then
$c_{ij}\ge0$ for all $i,j$.
\end{proposition}

\begin{proof}
Write $A=I+\lambda L$. Its off-diagonal entries are $-\lambda A^{\mathrm{adj}}_{ij}\le0$, its diagonal
entries are $1+\lambda d_i>0$, and $A_{ii}-\sum_{j\ne i}|A_{ij}|=1>0$, so $A$ is a nonsingular symmetric
$M$-matrix and $K=A^{-1}\ge0$ entrywise.

Non-negativity of $K$ is not by itself enough. A nonnegative positive-semidefinite kernel can induce coefficients of either sign. For
$K=\begin{pmatrix}1&0.5&0.1\\0.5&1&0.5\\0.1&0.5&1\end{pmatrix}$, which is positive definite and entrywise
positive, taking $S=\{1,2\}$ and $U=\{3\}$ gives $K_{US}K_{SS}^{-1}=(-0.2,\;0.6)$. What gives the
conclusion is the $M$-matrix structure of the precision. Partition $A$ by the unrevealed and revealed sets.
For a Gaussian with precision $A$ the conditional mean of $r_U$ given $r_S$ is $-A_{UU}^{-1}A_{US}r_S$, and
comparing it with $K_{US}K_{SS}^{-1}r_S$ gives
\[
   C(K)\;=\;K_{US}K_{SS}^{-1}\;=\;-A_{UU}^{-1}A_{US}.
\]
Principal submatrices of a nonsingular $M$-matrix are nonsingular $M$-matrices, so $A_{UU}^{-1}\ge0$
entrywise, while $A_{US}\le0$ entrywise. The product is therefore entrywise nonnegative.

The same holds with the soft-update nugget, where the coefficient is $K_{US}(K_{SS}+\sigma^2I)^{-1}$. This is
still the conditional mean under the working model with noisy observations: writing $Q=K^{-1}$ and letting
$E_S$ select the revealed coordinates, the update matrix is
$(Q+\sigma^{-2}E_SE_S^\top)^{-1}\sigma^{-2}E_S$, whose rows on $U$ coincide with
$K_{US}(K_{SS}+\sigma^2I)^{-1}$. Adding a nonnegative diagonal to $Q$ preserves the $M$-matrix property, so
this matrix is entrywise nonnegative as well. The argument is specific to the Laplacian kernel and does not
extend automatically to every family in the library.
\end{proof}

\begin{remark}
Only non-negativity is claimed. On irregular graphs the entries of $K$ need not decay monotonically in graph
distance, and the impossibility argument in the main text does not need them to. If $c_{ia}>0$ and $\Sig_{ia}<0$ then $\Delta_i<0$ for a single reveal by the Corollary. At $c_{ia}=0$ the
gain is zero.
\end{remark}

\subsection{Bayes-optimal linear propagation (Proposition 2)}

\begin{proposition}
Under the assumptions of Theorem~1 of the main paper, with a finite posterior mean covariance and a
reveal set fixed in advance, let $\bar V=\E[\Sig_{aa}\mid\D]+\sigma^2>0$. Then
$\E[\Delta_i(c)\mid\D]$ is maximised over $c\in\R$ at
$c^{B}=\E[\Sig_{ia}\mid\D]/\bar V$, and $\E[\Delta_i(c^{B})-\Delta_i(c)\mid\D]=\bar V(c-c^{B})^2$.
\end{proposition}

\begin{proof}
$\Delta_i(c)=2c\Sig_{ia}-c^2(\Sig_{aa}+\sigma^2)$ is affine in $\Sig$ for fixed $c$, so
$\E[\Delta_i(c)\mid\D]=2c\,\E[\Sig_{ia}\mid\D]-c^2\bar V$. This is the scalar identity with $\Sig$
replaced by its posterior mean, and the Corollary applies verbatim. The matrix statement follows the
same way from the Theorem, with $\bar V_S=\E[\Sig_{SS}\mid\D]+\sigma^2I$.
\end{proof}

\begin{remark}[Scope]
The optimum is over rules linear in $z_S$ whose coefficient is fixed before $z_S$ is observed. Under
model uncertainty a fully Bayesian predictor may update $\Pf(M_m\mid\D,z_S)$ and induce an effective
coefficient depending on the realised $z_S$. The proposition makes no claim against that class.
\end{remark}

\subsection{Regret controls ranking}

This is the bound the main text points to. It inherits the setting of Theorem~1 of the main paper:
centred residuals, with the reveal set and the coefficient matrix fixed before the residuals are
observed. Where $\sigma_{ij}=0$ the Gaussian tail term is taken to be zero. Write $\mathcal R(C)=\Delta(C^\star;S)-\Delta(C;S)$ for the
propagation regret and $\mu_i(C)=m_i+(Cz_S)_i$ for the post-propagation score.

\begin{proposition}
With $D=C-C^\star$, $u_i=D_iV_S^{1/2}$, $p_{ij}=(D_i-D_j)z_S$, $g_{ij}=|\mu_i(C^\star)-\mu_j(C^\star)|$ and
$\sigma_{ij}^2=\|u_i-u_j\|^2$, we have $\sum_{i<j}\sigma_{ij}^2\le|U|\,\mathcal R(C)$, and for every constant
$t>0$,
\[
   \E\,\#\{\text{misordered}\}\;\le\;\E\,\#\{g_{ij}<t\}\;+\;\sum_{i<j}2\Phi(-t/\sigma_{ij})
\]
if $z_S$ is Gaussian, and without that assumption the same bound holds with the last term replaced by
$|U|\,\mathcal R(C)/t^2$.
\end{proposition}

Write $D=C-C^\star$, $u_i=D_iV_S^{1/2}$ for the $i$-th row, and for a pair of unrevealed items
$p_{ij}=(D_i-D_j)z_S$, $g_{ij}=|\mu_i(C^\star)-\mu_j(C^\star)|$, $\sigma_{ij}^2=\|u_i-u_j\|^2$.

\emph{The pair identity.} $\sum_{i<j}\|u_i-u_j\|^2=\tfrac12\sum_{i,j}(\|u_i\|^2+\|u_j\|^2-2u_i^\top u_j)
=|U|\sum_i\|u_i\|^2-\|\sum_iu_i\|^2\le|U|\sum_i\|u_i\|^2=|U|\,\|DV_S^{1/2}\|_F^2=|U|\,\mathcal R(C)$,
the last equality being the quadratic regret identity of the main paper. Note $\E p_{ij}^2=(D_i-D_j)V_S(D_i-D_j)^\top=\sigma_{ij}^2$, which is
why $p$ must be defined on $z_S$: $V_S$ is the covariance of $z_S$, not of $r_S$.

\emph{The decomposition.} The propagated gap is $\mu_i(C)-\mu_j(C)=\Delta^\star_{ij}+p_{ij}$ with
$\Delta^\star_{ij}=\mu_i(C^\star)-\mu_j(C^\star)$. A misordering means
$\mathrm{sign}(\Delta^\star_{ij}+p_{ij})\neq\mathrm{sign}(\Delta^\star_{ij})$, which forces
$|p_{ij}|\ge|\Delta^\star_{ij}|=g_{ij}$. Hence for any constant $t>0$, \emph{pointwise in the draw},
\[
   \#\{\text{misordered}\}\;\le\;\#\{(i,j):g_{ij}<t\}\;+\;\#\{(i,j):|p_{ij}|\ge t\}.
\]

\emph{The tail.} Taking expectations, the second term is $\sum_{i<j}\Pf(|p_{ij}|\ge t)$. If $z_S$ is Gaussian
then $p_{ij}\sim\N(0,\sigma_{ij}^2)$ and this is $\sum_{i<j}2\Phi(-t/\sigma_{ij})$; without a distributional
assumption, Chebyshev gives $\sum_{i<j}\sigma_{ij}^2/t^2\le|U|\,\mathcal R/t^2$. The substitution is legitimate
only \emph{after} taking expectations and only because $t$ is a constant fixed in advance. Conditioning instead
on the realized gaps and writing $\sum2\Phi(-g_{ij}/\sigma_{ij})$ is invalid: $g_{ij}$ and $p_{ij}$ are
functions of the same draw of $z_S$, so the perturbation is not independent of the gap it must exceed, and that
expression comes out roughly twice the true count in simulation. \hfill$\square$

\begin{remark}
The proposition controls pairwise disagreement with the oracle \emph{linear score}. It is not a recall@$k$ guarantee against true labels. Such a guarantee would also require assumptions
linking the oracle scores to relevance, which we do not make.
\end{remark}

\section{Experimental detail}

\paragraph{Data.} \textsc{cora} ($2708$ nodes, $5278$ undirected edges) and \textsc{citeseer}
($3312$, $4536$) from the LINQS release. \textsc{texas} ($183$, $279$), \textsc{cornell} ($183$,
$277$), \textsc{wisconsin} ($251$, $450$) and \textsc{film} ($7600$, $26659$) from the Geom-GCN
release. \textsc{chameleon} ($2277$, $31371$) and \textsc{squirrel} ($5201$, $198353$) from the
Wikipedia article networks, with the traffic target discretised into five quantile classes; both contain
substantial numbers of nodes sharing a neighbourhood and a target, $46\%$ and $36\%$ respectively, as
discussed in the main text. Edges are symmetrised and self-loops removed.

\paragraph{Pools.} Each target class is a query. Pools are BFS neighbourhoods of up to $200$ nodes
seeded in the giant component. Uniform node samples are \emph{not} used. On \textsc{cora}, uniform samples of this size have mean degree
$0.29$, compared with $3.90$ in the full graph, so we use connected pools to retain enough graph structure
for the comparison. Up to $12$ training pools and $25$ test pools are drawn per graph. The overlap rule applies only to
training pools: test pools are drawn without it, so they may overlap one another and the training pools.

\paragraph{Mean and residuals.} The predictive mean is $\ell_2$-regularised logistic regression on node
features, cross-fitted over $5$ folds. Residuals from training queries are centred, and a new query's
mean is recalibrated by the pooled training bias, which uses no labels from that query. All covariance
estimates are leave-class-out.

\paragraph{What one independent likelihood unit is.} The model evidence factorises as
\[
   L_m(\theta;\D)=\prod_{q}\prod_{P}\ \N\bigl(r_{q,P};\,0,\,\Sigma_m(P;\theta)\bigr),
   \qquad p(\D\mid M_m)=\int L_m(\theta;\D)\,\pi_m(d\theta),
\]
over training queries $q$ and training pools $P$, so a single $(q,P)$ pair is one term and the
factorisation needs defending on both indices.

Overlapping pools can reuse observations in the product likelihood. We limit pairwise Jaccard overlap to
reduce this reuse, but the resulting pools are not independent. It was violated in an earlier
version of these experiments: a BFS ball of the requested size is the entire giant component on the
small graphs, so the twelve training pools had mean pairwise Jaccard overlap $1.000$ on \textsc{texas}
and \textsc{cornell} and $0.898$ on \textsc{wisconsin}, and identical data was entering the evidence
twelve times. Training pools are now accepted only when they overlap every previously accepted one by less than
$0.25$, which yields twelve pools on \textsc{cora}, \textsc{citeseer} and \textsc{film} and exactly one
on the three small graphs. These are near-disjoint rather than independent, and we describe them that
way throughout. The replicate count is reported beside every weight in
Table~1 of the main text, and the small-graph weights are correspondingly moderate. Test pools carry no
evidence and are drawn without the constraint.

Across queries, the residual vectors $r_{q,\cdot}$ for different target classes are not strictly
independent, since the class indicators sum to one and the leave-class-out residuals are therefore
dependent. Treating them as independent is an approximation, and we have not
established the direction of its effect on the posterior model weights. The sign of the raw residual
correlation does not by itself determine whether the product form over- or understates the evidence,
since that depends on the covariance of the likelihood scores. We note it as a limitation of the evidence calculation. We report the corresponding sensitivity check. Recomputing the evidence with a \emph{single} target class per pool, cycling over three
choices, changes the selected family on six of the eight graphs. The two exceptions are \textsf{cora} and \textsf{squirrel}, both of which put essentially all mass on $M_+$.
\textsf{film} also selects $M_+$, at $0.54$, and does switch, so retaining the selection is not simply a
property of preferring $M_+$. This check changes both the dependence among likelihood terms and the amount of evidence, so it does not
isolate either effect, and we did not evaluate predictive performance under the reduced likelihoods
(\texttt{scripts/bayes\_covariance2.py} with \texttt{--one-class} $\in\{0,1,2\}$).

\paragraph{Covariance parameterisation.} $K^{G}(\lambda)$ is used in correlation form, so
$aI+bK^{G}$ would carry diagonal $a+b$ while the distance-profile family carries $a$, its basis
$H_h$ starting at $h\ge1$. The families would then differ on the entry that dominates a Gaussian
likelihood. We therefore write the smoothing families as $aI+b(K^{G}(\lambda)-I)$, so every candidate
has diagonal exactly $a$ and $b$ describes off-diagonal structure alone. An earlier version of these
experiments did not, and the smoothing families carried up to $1.30$ times the null variance, which
materially changed both the selected shapes and the family-recovery experiment.

\paragraph{Model fitting.} Priors: $\N(0,(a/2)^2)$ on the scale parameter of each non-null family,
truncated to the admissible set $\{\theta:\Sigma_m(\theta)\succeq\epsilon I\}$ and renormalised there,
so that a range value whose admissible interval is narrow is not penalised merely for having more of
the untruncated prior outside it. Uniform on the range grids $\lambda\in\{0.3,1,3,10\}$ and $\rho\in\{-0.85,\dots,0.85\}$ ($19$ points), and uniform
over the four model families. Given the range parameter the remaining integral is one-dimensional and
is evaluated by trapezoid quadrature on $161$ points over that interval, with $\epsilon=10^{-3}a$. The same quadrature returns the
within-model posterior mean of the scale, and the range parameter is averaged over its grid weighted by
each value's marginal likelihood, so the covariance entering the Bayes coefficient is a posterior mean in the shape parameters throughout. The residual variance $a$ is a plug-in estimate from cross-fitted residuals on the training queries, using
no labels from the evaluation query, which makes the construction empirical Bayes in that one nuisance
parameter. Acquisition is UCB with
$\beta=0.7$, $B=20$ reveals, the regularisation nugget $\tau=0.05$.

\paragraph{Baselines.} \textsc{c\&s} follows the correction step of Huang et al.\ with $\alpha=0.8$ and
$50$ diffusion steps on the symmetrically normalised adjacency. We report it both with and without the
autoscale step: autoscale is designed for a transductive regime with many labels and amplifies badly
when only $20$ of $200$ items are revealed (\textsc{cora} MSE $0.1130$), so the plain correction step
is the stronger variant and is the one used in the main comparison.

\paragraph{Reveals, updates and intervals.} The budget $B=20$ is the \emph{total} number of reveals in a
pool. Acquisition begins from an empty revealed set, so no pool node is labelled in advance, and the
predictive mean is cross-fitted outside the pool and therefore uses no pool labels. The update is
$m+K_{:S}(K_{SS}+\tau I)^{-1}(y_S-m_S)$, applied without truncation to $[0,1]$; squared error and recall
are computed on that unmodified linear update, which is the object the identity evaluates. \textsc{feat}
is the linear feature Gram matrix in correlation form with a $10^{-6}$ ridge, and \textsc{gmrf} is
$K^{G}(\lambda)$ at $\lambda=1$, fixed rather than tuned. The admissible scale interval is intersected
across training pools, so one interval is used that holds $\Sigma\succeq\epsilon I$ on every training pool
rather than a separate interval per pool. Reported intervals are paired bootstrap percentile intervals at
the $2.5$ and $97.5$ per cent points, with $4000$ replicates resampling test pools within each graph.
Fitted without the positive-semidefinite constraint, the free distance profile has minimum eigenvalue
$-0.0085$ on \textsc{cora} and $-0.0143$ on \textsc{cornell}, which is why the constraint is imposed.

\paragraph{Family-recovery experiment.} Residuals are generated on the real \textsc{cora} graph under each known
shape, at $6$ pools $\times\,6$ queries with $12$ replicates. Binary residuals are obtained by
thresholding a latent Gaussian at base rate $0.15$ and centring, so that the induced residual covariance preserves the sign and ordering of the target but generally
changes its scale and its covariance-family form.

\begin{table}[!ht]
\centering
\caption{Shape recovery. Rows indicate the generating
family and columns the selected family. Under binary
residuals, the row label names the latent Gaussian
construction, so the observed binary covariance need
not lie in that family. The binary columns therefore
measure robustness to that construction rather than
literal recovery. There are $12$ replicates, with
$6$ pools $\times\,6$ queries, at the scale of the
real experiments.}
\label{tab:calib}
\begin{tabular}{lrrrr@{\hskip 1.4em}rrrr}
\toprule
& \multicolumn{4}{c}{Gaussian residuals}
& \multicolumn{4}{c}{Binary residuals}\\
\cmidrule(r){2-5}\cmidrule(l){6-9}
generating family
& $M_0$ & $M_+$ & $M_-$ & $M_H$
& $M_0$ & $M_+$ & $M_-$ & $M_H$\\
\midrule
$M_0$ & $\bm{.92}$ & $.08$ & $.00$ & $.00$
      & $\bm{.83}$ & $.00$ & $.08$ & $.08$\\
$M_+$ & $.00$ & $\bm{1.00}$ & $.00$ & $.00$
      & $.00$ & $\bm{1.00}$ & $.00$ & $.00$\\
$M_-$ & $.00$ & $.00$ & $\bm{1.00}$ & $.00$
      & $.50$ & $.00$ & $\bm{.50}$ & $.00$\\
$M_H$ & $.00$ & $1.00$ & $.00$ & $\bm{.00}$
      & $.00$ & $.25$ & $.00$ & $\bm{.75}$\\
\bottomrule
\end{tabular}
\end{table}

\section{Alternatives examined}

Two alternatives were examined and are not used in the reported method.

\emph{Sequential within-query updating} changes the family weights by at most $0.01$ over $B=20$ reveals.
This is the quantity the scope remark of the main paper appeals to.

\emph{Ranking candidate kernels by the identity} is circular. Evaluating $\Delta$ at an estimated $\Sig$ is
maximised by construction at $K=\hat\Sig$, so it selects the plug-in on all eight graphs. This is why the
method compares marginal likelihoods instead.

\section{Effect of covariance-family choice on propagation risk}\label{sec:action}
Table~\ref{tab:calib} evaluates recovery of the generating family, a task the method is never asked to
perform. We also evaluate the resulting propagation coefficients. The procedure chooses a coefficient
$c^{B}$, and Theorem~1 of the main paper gives the excess squared error of any $C$ relative to the
true-covariance oracle as
$\|(C-C^\star)V^{1/2}\|_F^2$ with $V=\Sigma^\star_{SS}+\sigma^2I$ and $C^\star=\Sigma^\star_{US}V^{-1}$.
In simulation $\Sigma^\star$ is known, so every rule is evaluated on that scale. Under the binary law
$\Sigma^\star$ is the covariance of the \emph{latent} Gaussian rather than of the thresholded residuals the
estimators see, and the residuals are additionally rescaled to unit mean square while $\sigma^2$ is held
fixed. The binary column therefore compares rules estimating a binary covariance against an oracle defined
on the latent scale, and its absolute levels should not be read as fractions of an attainable gain. The
Gaussian column is unaffected. Under Gaussian residuals we report regret as a fraction of $\mathrm{regret}(0)$, the cost of refusing to
propagate, so $0$ is optimal and $1$ discards the whole available gain. The binary column is not a fraction
of an attainable gain and is reported only as a discrepancy from the latent-scale oracle.

Under a true $M_0$ the optimum is $C^\star=0$, so $\mathrm{regret}(0)=0$ and the fraction is $0/0$;
that row is omitted.
On the three truths carrying dependence the \textsc{never} arm scores exactly $1.000$, which is an
arithmetic check on the normalisation.

\begin{table}[h]\centering\small
\caption{Regret of the propagation coefficient as a fraction of the cost of never propagating for the Gaussian column only; the binary column measures
discrepancy from a latent-scale oracle and is not a fraction of attainable predictive gain, on the
truths that carry dependence. Paired within replicate; $12$ replicates, $6$ pools, $5$ reveal splits,
$\sigma^2=1$ (the sweep over $\sigma^2\in\{0.1,0.5,1\}$ moves the third digit).
\texttt{scripts/cov\_action\_regret.py}.}\label{tab:action-supp}
\begin{tabular}{lcc}
\toprule
paired difference & gaussian & binary (latent oracle) \\ \midrule
value of knowing the family & $+0.008$ & $+0.060$ \\
 & {\footnotesize$[.005,.011]$} & {\footnotesize$[.031,.092]$}\\
cost of assuming a family & $+0.323$ & $+0.186$ \\
 & {\footnotesize$[.187,.481]$} & {\footnotesize$[.095,.290]$}\\
\bottomrule
\end{tabular}
\end{table}

Two differences answer the question, both paired within replicate on identical data: the value of
knowing the family, $\mathrm{regret}(\textsc{bma})-\mathrm{regret}(\textsc{oracle-fam})$, where the
oracle is handed the true family and still estimates its parameter; and the cost of assuming one,
$\mathrm{regret}(\textsc{always-}M_+)-\mathrm{regret}(\textsc{bma})$.

\paragraph{A comparison promoted after inspecting results.} We initially planned to compare regret on
misidentified replicates against regret on identified ones. Because identification success differed strongly
by generating family ($M_+$ recovered in every replicate, $M_H$ in none), that comparison confounded family
with identification. We therefore report the paired comparison above, originally a secondary analysis. This
choice was made after inspecting the results. Produced by \texttt{scripts/cov\_action\_regret.py}.

\end{document}